%% file: main.tex
\documentclass[11pt]{article}
\usepackage{amsmath, amsthm, amssymb, amsfonts}
\usepackage[alphabetic]{amsrefs}
\usepackage{mathrsfs}
\usepackage{thmtools}
\usepackage{graphicx}
\usepackage[margin=1in]{geometry}
\usepackage{tikz-cd}
\usepackage[dvipsnames]{xcolor}
\usepackage{import}
\usepackage{quiver}
\usepackage{authblk}
\usepackage{titlesec}
\usepackage{subcaption}
\usepackage{svg}

\usepackage[colorlinks=true]{hyperref}

\def\l{\left}
\def\r{\right}
\def\lr{\leftrightarrow}

\def\N{\mathbb{N}}
\def\Z{\mathbb{Z}}

\def\a{\alpha}
\def\b{\beta}
\def\g{\gamma}

\def\vp{\varphi}
\def\map{\text{Map}}
\def\Aut{\text{Aut}}
\def\End{\text{End}}

\def\squig{\leftrightsquigarrow}

\newtheorem{theorem}{Theorem}[section]
\newtheorem{ntheorem}{Theorem}
\newtheorem{definition}[theorem]{Definition}

\newtheorem{ncorollary}[ntheorem]{Corollary}
\newtheorem{proposition}[theorem]{Proposition}

\newtheorem{lemma}[theorem]{Lemma}

\titleformat*{\section}{\scshape\center\large}

\renewcommand{\abstractname}{\textmd{\textsc{Abstract}}}
\renewenvironment{abstract}
 {\small
  \begin{center}
  \abstractname\vspace{0pt}\vspace{0pt}
  \end{center}
  \quotation}
 {\endquotation}

\title{\large\textsc{Hatcher--Thurston complex for surfaces with non-planar ends}}
\author{\normalsize Manvendra Somvanshi}
\date{}

\begin{document}

\maketitle
\begin{abstract}
In this paper, for each $k\in \N$, we define a complex $\Gamma_k(S)$ for an infinite-type surface $S$ with non-planar ends, which serves as an analog of the Hatcher--Thurston complex for the infinite-type setting. We show that $\Gamma_k(S)$ is connected, simply connected, and that the automorphism group of $\Gamma_k(S)$ is isomorphic to the extended mapping class group.
\end{abstract}
\section*{Introduction}\label{intro}
\addcontentsline{toc}{section}{Introduction}
All surfaces in this paper will be connected, orientable, without any punctures or marked points, and possibly with non-empty compact boundary. The mapping class group, $\map(S)$, is the group of all orientation preserving homeomorphisms of the surface $S$ up to isotopy. The extended mapping class group, $\map^*(S)$ is the group of all homeomorphisms of $S$ up to isotopy. When the fundamental group $\pi_1(S)$ of the surface $S$ is finitely generated, we say that $S$ is of finite-type; otherwise $S$ is infinite-type.\\

Recently, there has been an effort in the literature to generalize results about curve graphs and complexes from the finite-type case to the infinite-type case. In \cite{bavard}, the authors show that the automorphism group of the curve complex, $ \mathcal{C}(S)$, of an infinite-type surface with empty boundary is geometric, that is, every automorphism of $ \mathcal{C}(S)$ is induced by the action of $\map^*(S)$, generalizing the result of \cites{ivanov, luo1999, korkmaz} for finite-type surfaces. Independently, \cite{hernandez} also proved that the automorphism groups of the curve complex, non-separating curve complex, and the Schmutz graph are all isomorphic to each other and to $\map^*(S)$ for surfaces of infinite-type with no planar ends and finitely many boundary components. Moreover, \cite{hernandez, bavard} show that these complexes can differentiate infinite-type surfaces. In \cite{branman}, the author generalizes the result that automorphism group of the pants complex is isomorphic to the extended mapping class group, a result originally proved in \cite{margalit04} for finite-type surfaces. \\

In this paper, we generalize the Hatcher--Thurston complex to the setting of infinite-type surfaces with possible compact boundary and non-planar ends. The Hatcher--Thurston complex or Cut-system complex, $ \mathcal{HT}(S)$, is a 2-dimensional CW-complex associated to surfaces of finite-type in which the vertices are maximal cut-systems of the surface, edges correspond to an elementary move between cut-systems, and 2-cells are attached to triangles, rectangles, and pentagons as described in Section \ref{pre}. This cut-system complex was first introduced in \cite{hatcher} to obtain a presentation for the mapping class group of finite-type surfaces. In \cite{hatcher}, the authors use a Morse-theoretic proof to show that the complex is connected and simply connected. Later, Wajnryb gave a combinatorial proof of this in \cite{wajnryb}. In \cite{irmak2004}, the automorphisms of the Hatcher--Thurston complex were shown to be geometric. In this paper, the complex $ \mathcal{HT}(S)$ is generalized to a countable collection of complexes $\l(\Gamma_k(S)\r)_{k\in \N}$ whose vertices are given by cut-systems with exactly $k$ non-separating curves, as opposed to maximal cut-systems, with $1-$cells and $2-$cells defined in a similar fashion to $ \mathcal{HT}(S)$. In case that $S$ is a finite-type surface of genus $g$, the collection $\{\Gamma_k(S)\}$ is finite with $1\leq k \leq g$ and $\Gamma_g(S) = \mathcal{HT}(S)$. Precise definitions of these complexes are given in Section \ref{pre}. Since the construction of $1$-cells and $2-$cells of the pants complex in \cite{branman} is similar to that of the Hatcher-Thurston complex, generalization of the pants complex to infinite-type surfaces leads to similar issues with connectedness as the generalization of the Hatcher-Thurston complex. In \cite{branman} these issues are resolved by introducing a family of new topologies on the vertex space and extending it to the enitre pants space. As will be discussed in Section \ref{pre}, we resolve this issue, among others, by instead defining the sequence $(\Gamma_k(S))_{k\in \N}$ of complexes discussed above. In Section \ref{con}, we show the following result.
\begin{ntheorem}\label{thm:simply} 
  Let $S$ be a surface of infinite-type with non-planar ends. Then $\Gamma_k(S)$ is connected and simply connected for each $k\in \N$.
\end{ntheorem}
The proof of this result closely follows the combinatorial proof of \cite{wajnryb} with alterations made when required for the infinite-type case. The proof proceeds by induction on $k$. We first show that the $1-$skeleton, $\Gamma^1_k(S)$, is connected and that its diameter is bounded between $k$ and $8k-4$ using induction on $k$. To further prove that $\Gamma_k(S)$ is simply connected, we introduce the notions of radius of a path in $\Gamma_k(S)$ and segments, following the proof by Wajnryb. The remaining proof can then be divided into three steps. It is first shown in Proposition \ref{prop:base-case} that $\Gamma_1(S)$ is simply-connected by first showing that closed paths with radius $1$ are contractible. Note that this argument uses the assumption that $S$ has at least one non-planar end. Continuing the induction on $k$, the second step shows that any closed 0-radius path is contractible in $\Gamma_k(S)$, see Proposition \ref{lem:radius-zero}, by induction on the number of segments in the path. Finally, again using the result about radius $0$ paths and the assumption of a non-planar end the theorem statement follows.\\

\textbf{Remark.} At multiple points throughout the proof of Theorem \ref{thm:simply} the existence of at least one non-planar end is invoked, leading one to ask whether this is a necessary condition. In other words, is $\Gamma_k(S)$ simply connected for all $1\leq k \leq g$ when $S$ is a genus $g$ surface? Since $\Gamma_g(S) = \mathcal{HT}(S)$ for finite-genus surfaces, as mentioned earlier, it follows that $\Gamma_g(S)$ is simply-connected. But the proof of this, in \cite{wajnryb, hatcher}, makes use of the maximality of the cut-systems with $g$ curves in a genus $g$ surface. To the author's knowledge, this question does not seem to be answered in the literature.\\

In Section \ref{aut}, we show that automorphisms of $\Gamma_k(S)$ are geometric:
\begin{ntheorem}\label{thm:geometric}
  Let $S$ be a surface of infinite-type with no planar ends and empty boundary. Then $\text{Aut}(\Gamma_k(S))$ is isomorphic to $\map^*(S)$ for each $k\in \N$.
\end{ntheorem}
We prove this by constructing an action of $\text{Aut}(\Gamma_k(S))$ on the Schmutz graph, $ \mathcal{G}(S)$, as done by \cite{irmak2004}. The proof in \cite{irmak2004} is almost entirely transferable to the infinite-type case, but parts of it are reproduced here for the sake of completeness. As a direct corollary to Theorem \ref{thm:geometric} we have the following:
\begin{ncorollary}\label{cor:iso}
  Let $S,S'$ be surfaces of infinite-type with no planar ends and empty boundary. Let $\vp:\Gamma_k(S)\to \Gamma_k(S')$ be an isomorphism. Then $S$ is homeomorphic to $S'$.
\end{ncorollary}
The Hatcher--Thurston complex is an important tool in the study of mapping class group of finite-type surfaces. By studying the stabilizers of vertices, edges, and cells in $ \mathcal{HT}(S)$ of the action of $\map(S)$ Wajnryb gave a simple presentation of $\map(S)$ in \cite{wajnryb2}, which was inspired by a similar approach to the problem by Harer in \cite{harer}. Harer in the same paper uses a simpler version of the Hatcher--Thurston complex to compute the second cohomology of the mapping class group. These methods are not directly transferable to big mapping class group as they are not finitely generated, in fact they are not even countably generated. In \cite{patel}, the authors give topological generators for the pure mapping class group $\text{PMap}(S)$, the subgroup of $\map(S)$ which are end-preserving, in terms of Dehn Twists and handle shift maps. For infinite-type surfaces with countably many non-planar ends it may be possible to use the results presented in this paper to determine relations in $\text{PMap}(S)$ between Dehn Twists and handle shift maps by applying the techniques of Wajnryb and Harer.
\vspace{1em}
\newline
\textbf{Outline.} In Section \ref{pre} the main objects of this paper are defined and conventions and notations which are used throughout the paper are discussed. Section \ref{con} is dedicated to the proof of Theorem \ref{thm:simply} and Section \ref{aut} to the proof of Theorem \ref{thm:geometric} and Corollary \ref{cor:iso}.
\vspace{1em}
\newline
\textbf{Acknowledgements.} The author would like to thank Sumanta Das for posing the question which led to this paper and his comments for improving the quality of the paper. The author also extends gratitude to Dan Margalit and Tyrone Ghaswala, whose comments helped improve the paper.

\section{Preliminaries}\label{pre}
\textbf{Ends of a Surface.} Let $S$ be a surface of infinite-type and let $\{K_i\}_{i\in \N}$ be a strictly increasing chain of compact sets ordered by inclusion which covers $S$. Then an \textit{end}, $e$, of the surface $S$ is a strictly decreasing chain $\{U_i\}_{i\in \N}$, where each $U_i$ is a connected component of $S-K_i$. Let $\End(S)$ denote the set of all ends of $S$. Given an open set $U\subset S$ with compact boundary define:
$$U^* = \{e = \{U_i\}_{i\in\N}\in \End(S)\ |\ U_i\subset U\ \text{for large enough}\ i\in\N\}$$
Endow $\End(S)$ with the topology generated by $\{U^*\ |\ U\ \text{open subset of}\ S\ \text{with compact boundary}\}$. For any surface, $S$, the space of ends is homeomorphic to a closed subset of the Cantor set, \cite{ahlfors}. An end, $e=\{U_i\}_{i\in \N}$, is said to be \textit{planar} if for some $i\in \N$, $U_i$ can be embedded in the plane. Otherwise, the end is said to be \textit{non-planar} or \textit{accumulated by genus}. Denote the subspace of non-planar ends to be $\End_{\infty}(S)$. The classification theorem for infinite-type surfaces states that $S, S'$ are homeomorphic if they have the same genus and number of boundary components and there is a homeomorphism $\End(S)\to \End(S')$ which restricts to a homeomorphisms of the subspace of non-planar ends, $\End_{\infty}(S)\to \End_\infty(S)$, \cite{Richards}. Throughout this paper, the surface $S$ is assumed to have at least one non-planar end.
\vspace{1em}
\newline
\textbf{Curves and Arcs.} By a curve $\a$ in the surface $S$ we mean a simple closed curve, that is, an embedding of $S^1$ in $S$, which is not null-homotopic and is not homotopic to a boundary component of $S$. We often consider isotopy classes of curves, which we will denote by the latin symbol corresponding to the greek symbol of a representative curve, i.e. $a = [\a],\ b= [\b],\ c=[\g],$ etc. The \textit{geometric intersection number}, $i(a,b)$, of a pair of isotopy classes of curves $a,b$ is defined as the infimum of the set $\{|\a\cap \b|\ |\ \a\in a,\ \b\in b\}$. Curves $\a$ and $\b$ are said to be in \textit{minimal position} if $|\a\cap \b| = i(a,b)$. Just like in the finite-type case, the curves $\a, \b$ are in minimal position if and only if they do not form any bigons, \cite{farb}. Given two isotopy classes of curves $a,b$ there are always representatives $\a,\b$, respectively, which are in minimal position. We often abuse notation and call both $\a$ and $a$ as curves. The isotopy class, $a$, of a curve, $\a$, is said to be \textit{separating} if $S-\a$ is not connected, otherwise $a$ is said to be \textit{non-separating}. Note that if $i(a,b) = 1$ then both $a,b$ must be non-separating.\\

Let $I = [0,1]$. An arc $\tau$ on $S$ is an embedding of $\tau:I\to S$ so that $\tau(\partial I) \subset \partial S$. An arc is said to be essential if it is not isotopic to a boundary or point in $S$, otherwise it is said to be non-essential. Just like curves, we think of arcs in terms of their isotopy classes. The definitions of separating, non-separating, etc can be extended to arcs.
\vspace{1em}
\newline
\textbf{Mapping class group.} Let $\text{Homeo}(S)$ be the group of self-homeomorphisms of any surface $S$ then the extended mapping class group $\map^*(S)$ is defined as $\pi_0(\text{Homeo}(S))$, and elements of $\map^*(S)$ are called mapping classes and are denoted $[F]$ where $F\in \text{Homeo}(S)$ is a representative homeomorphism. Endowing $\text{Homeo}(S)$ with the compact-open topology induces a topology on $\map^*(S)$. When $S$ is finite-type this is just the discrete topology. This topology is non-trivial when $S$ is of infinite-type. Given a mapping class $[F]$ it induces a map $F_*$ on the collection of isotopy classes of curves: $F_*(a) = F(a)$. It is easily checked that this is well defined. Dehn twist $T_a$ about a curve $a$ is a mapping class which is represented by a homeomorphism with support inside an annular neighbourhood $A$ containing $\a$ given by $(\theta, t)\mapsto (\theta+2\pi t, t)$.
\vspace{1em}
\newline
\textbf{Curve Complexes and Graphs.} For any surface $S$ we define the curve complex $ \mathcal{C}(S)$ as the simplicial complex with vertices being the collection of isotopy classes of curves and the vertices $a_1,\cdots, a_n$ in $ \mathcal{C}(S)$ form an $n-$simplex if $i(a_i,a_j)=0$. The non-separating curve complex, $ \mathcal{N}(S)$, is the subcomplex of $ \mathcal{C}(S)$ with vertex set consisting of only of non-separating curves. The Schmutz graph, $ \mathcal{G}(S)$, is the graph with the same vertex set as $ \mathcal{N}(S)$, but two vertices have an edge between them if $i(a,b)=1$. It is a well-known fact that $ \mathcal{C}(S), \mathcal{N}(S),$ and $ \mathcal{G}(S)$ are connected for any surface, $S$. In the case when $S$ is of infinite-type all of these complexes have finite diameter, \cite{hernandez}. The action of $\map^*(S)$ on curves naturally induces an action on these curve complexes via automorphisms.
\begin{theorem}[Theorem 1, \cite{hernandez}]\label{thm:her}
 Let $S$ be an infinite-type surface with possible compact boundary and no planar ends, that is, $\End(S) = \End_\infty(S)$. Then the action of $\map^*(S)$ on $\Gamma$ induces an isomorphism to $\Aut(\Gamma)$, where $\Gamma$ is one of $ \mathcal{C}(S),\ \mathcal{N}(S),\ \mathcal{G}(S)$.
\end{theorem}
Furthermore, it was shown in the same paper that $\Aut( \mathcal{C}(S))$ is rigid, that is, every isomorphism $\Aut( \mathcal{C}(S)) \to \Aut( \mathcal{C}(S'))$ is induced by a homeomorphism $S\to S'$.
\vspace{1em}
\newline
\textbf{Cut Systems and Hatcher--Thurston Complex.} A cut system on a surface $S$ is a disjoint collection of isotopy classes of non-separating curves, $(a_0,\cdots, a_k)$, so that $S-\bigcup \a_i$ is connected. Now, we recall the definition of Hatcher--Thurston complex on finite-type surfaces. Let $\Sigma_g$ be a finite-type surface with genus $g$ with possible compact boundary. Let $ \mathcal{HT}^0(\Sigma_g)$ be the collection of maximal elements of the poset of cut systems on $\Sigma_g$ ordered by inclusion. Every cut-system in $ \mathcal{HT}^0(\Sigma_g)$ has exactly $g$ curves. Cut systems $v,w\in \mathcal{HT}^0(\Sigma_g)$ are said to be related by an elementary move if $v-w = \{a\}$, $w-v = \{b\}$, and $i(a,b) = 1$, we denote this as $v \leftrightarrow w$. Attach a $1-$cell between two points $ v,w\in \mathcal{HT}^0(\Sigma_g)$ whenever they are related by an elementary move to obtain the graph, $ \mathcal{HT}^1(\Sigma_g)$. There are three types of simple cycles we consider:
\begin{enumerate}
  \item \textbf{Triangles:} Consider the vertices $v = (b_0, a_1, \cdots, a_{g-1})$, $w=(b_1, a_1,\cdots, a_{g-1})$, and $u=(b_2, a_1,\cdots, a_{g-1})$ so that  $i(b_i,b_j) = 1$. Then the cycle, $v \lr w \lr u \lr v$, is called a triangle. 
  \item \textbf{Rectangles:} Let $b_0,\ b_1,\ c_0,$ and $c_1$ be non-separating curves in $\Sigma_g$ with $i(b_0,b_1) = i(c_0,c_1) = 1$ and all other pairs being disjoint. If $v_{ij} = (b_i, c_j, a_2,\cdots, a_{g-1})\in \mathcal{HT}^1(\Sigma_g)$ then the cycle, $v_{00}\lr v_{01}\lr v_{11} \lr v_{10} \lr v_{00}$, is called a rectangle.
  \item \textbf{Pentagons:} Let $b_0,\cdots, b_4$ be non-separating curves in $\Sigma_g$ so that $i(b_i,b_{i+1}) = 1$, where the indices are modulo $5$. The vertices $v_i = (b_i, b_{i+2}, a_2,\cdots, a_{g-1})$ form a cycle, $v_0\lr v_2\lr v_4\lr v_1 \lr v_3\lr v_0$, called a pentagon.
\end{enumerate}
Attach a 2-cell along every cycle of type 1,2, and 3 in $ \mathcal{HT}^1(\Sigma_g)$ to construct the Hatcher--Thurston complex, $ \mathcal{HT}(\Sigma_g)$.
\begin{theorem}[Theorem 1.1, \cite{hatcher}]
  $ \mathcal{HT}(\Sigma_g)$ is connected and simply connected.
\end{theorem}
Consider now, an infinite-type surface $S$ with a non-planar end. There are two obvious issues with blindly generalizing the Hatcher--Thurston complex for $S$. Firstly, not all chains of cut systems have a maximal element when ordered by inclusion, for instance there exists a countable disjoint collection of non-separating curves in the Loch Ness monster such that the union of all the curves is separating but no finite subcollection is separating. Secondly, two cut systems with infinitely many distinct curves will not be connected by finitely many elementary moves. Thus we define the following complex by limiting cut systems to finitely many curves.
\begin{definition}
  Let $S$ be an infinite-type surface with non-planar ends, $k\in \N$, and let $\Gamma^0_k(S)$ be the collection of cut systems on $S$ with exactly $k$ curves. Let $\Gamma^1_k(S)$ be the graph obtained by attaching a 1-cell between two cut systems in $\Gamma^0_k(S)$ whenever they are related by an elementary move. The generalized Hatcher--Thurston complex, $\Gamma_k(S)$, is obtained by attaching 2-cells to every cycle of type $1,\ 2,$ or $3$.
\end{definition}
Note that $\Gamma^1_1(S)$ is just the Schmutz graph $ \mathcal{G}(S)$. Since the Schmutz graph is connected and $\text{diam}( \mathcal{G}(S))\leq 4$ we conclude that $\Gamma^1_1(S)$ is connected with finite diameter. This will be useful in the proof of connectedness of $\Gamma_k(S)$.

\section{Connectedness and Simple Connectedness of $\Gamma_k(S)$}\label{con}
In this section we prove Theorem \ref{thm:simply}. 
\begin{lemma}\label{lem:common}
  Let $v,w$ be cut systems in $\Gamma_k(S)$ so that $v$ and $w$ have all but one curve in common. Then there is a path of length at most $4$ from $v$ to $w$ in $\Gamma^1_1(S)$.
\end{lemma}
\begin{proof}
  Let $v-w = \{a\}$ and $w-v=\{b\}$ and let $R$ be the surface obtained by cutting $S$ along the representatives of all the common curves, $c_1,\cdots, c_{k-1}$, of $v$ and $w$. There is a path $a\lr d_1\lr\cdots\lr d_{n}\lr b$ in $\Gamma^1_1(R) = \mathcal{G}(R)$ with $n\leq 3$. Since $d_i$ are non-separating curves in $R$ it follows that $v_i = (d_i, c_1,\cdots, c_{k-1})$ is a cut system. Thus, $v\lr v_1\lr\cdots \lr v_n\lr w$ is the desired path.
\end{proof}
\begin{proof}[Proof of first half of Theorem \ref{thm:simply}]
  Let $v,w\in \Gamma^0_k(S)$. We use induction on the number of curves $\ell = |v-v\cap w|$. Lemma \ref{lem:common} shows that for the base case $\ell=1$ there is a path from $v$ to $w$. Suppose now that $v-w = \{a_0,\cdots, a_{\ell-1}\}$, $w-v=\{b_0,\cdots, b_{\ell-1}\}$, and $v\cap w = \{c_0,\cdots, c_{k-\ell-1}\}$. Let $K\subset S$ be a compact sub-surface containing all the curves in $v$ and $w$ and let $d$ be a non-separating curve in some component of $S-K$. Then, using induction, there is a path from $v' = (d,a_1,\cdots,a_{\ell-1},c_0,\cdots,c_{k-\ell-1})$ to $w'=(d,b_1,\cdots,b_{\ell-1},c_0,\cdots,c_{k-\ell-1})$. Using Lemma \ref{lem:common}, there is a path from $v$ to $v'$ and $w$ to $w'$.
\begin{figure}[h]
\[\begin{tikzcd}
	v && w \\
	\\
	{v'} && {w'} & {}
	\arrow[squiggly, tail reversed, from=1-1, to=3-1]
	\arrow[squiggly, tail reversed, from=1-3, to=3-3]
	\arrow[squiggly, tail reversed, from=3-1, to=3-3]
\end{tikzcd}\]
\caption{Constructing the path from $v$ to $w$.}
\end{figure}
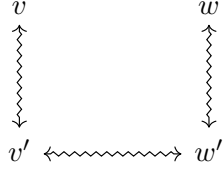
Using the construction above and the fact that the length of the path in Lemma \ref{lem:common} is at most $4$, it follows that the diameter of $\Gamma^1_k(S)$ is at most $8k-4$.
\end{proof}
If $v,w$ are two cut systems in $\Gamma_k(S)$ so that all of the curves in $v\cup w$ are distinct then any path in $\Gamma^1_k(S)$ between $v$ and $w$ has length at least $k$. This means that the diameter of $\Gamma^1_k(S)$ has a lower bound $k$.\\

In order to prove that $\Gamma_k(S)$ is simply-connected we define the following notion of radius of a path, and segments.
\begin{definition}
  The radius of a path $p = v_0 \lr \cdots \lr v_n$ in $\Gamma^1_k(S)$ with respect to a non-separating curve, $a\in v_j$ for some $j$, is defined as $\max_{v_i}\l(\min_{b\in v_i}i(a,b)\r)$. If $a_0,\cdots, a_n$ are non-separating curves in $S$ then $p$ is a $a_0,\cdots,\ a_n-$segment if $a_0,\cdots,a_n\in v_i$ for each vertex $v_i$ in the path $p$.
\end{definition}
First we show some propositions and lemmas in order to prove that $\Gamma_1(S)$ is simply-connected. The first one we will require is a special case of a well-known result about Dehn twists and intersection numbers.
\begin{proposition}[Proposition 3.4, \cite{farb}]\label{prop:dehn-int}
  Let $a,\ b,\ c$ be arbitrary isotopy classes of curves in $S$ and let $T_a$ be the Dehn twist about $a$. Then for any $k\in \Z$,
  $$\bigg| i(T_a^{k}(b), c) - |k| i(a,b)i(a,c) \bigg| \leq i(b,c)$$
$\square$
\end{proposition}
Note that vertices in $\Gamma_1(S)$ are just non-separating curves. A path is just a finite sequence, $a_0\lr a_1\lr \cdots \lr a_n$, of non-separating curves. We first show the following lemma.
\begin{proposition}[Lemma 9, \cite{wajnryb}]\label{prop:four-cycles}
  Let $p = a_0\lr a_1\lr a_2\lr a_3 \lr a_0$ be a closed path in $\Gamma_1(S)$ so that $i(a_1,a_3)=0$. Then $p$ is null-homotopic.
\end{proposition}
\begin{proof}
  Note that for some $s \in \{1,-1\}$ the Dehn twist $T_{a_1}^s$ satisfies $i(T^s_{a_1}(a_2),a_0) < i(a_2,a_0)$, for any $k\in \Z$ the intersection $i(T_{a_1}^k(a_2), T^{k\pm 1}_{a_1}(a_2)) = 1$, and for any $k\in \Z$ the intersection $i(T^k_{a_1}(a_2), a_1) = i(T^k_{a_1}(a_2), a_3)=1$, using the identity in Proposition \ref{prop:dehn-int}. This means that for some $s\in \{1,-1\}$ and $k>0$ the sequence of closed paths $p_k = a_0\lr a_1\lr T^{sk}_{a_1}(a_2)\lr a_3 \lr a_0$ is such that $i(T^{sk}_{a_1}(a_2),a_0)<i(T^{sk-s}_{a_1}(a_2),a_0)$ and $p_k$ and $p_{k+1}$ are homotopic as $p_{k+1}-p_k$ can be split into two triangle paths. Let $p_N$ be the path where the first and third curves are disjoint, then it follows by induction that $p_0=p$ and $p_N$ are homotopic. Thus without loss of generality we now assume that $i(a_0,a_2)=0$ in $p$.
  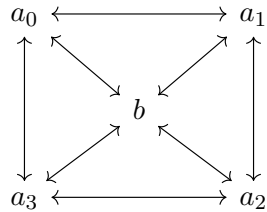
\begin{figure}[h]
    \[\begin{tikzcd}
	{a_0} && {a_1} \\
	& b \\
	{a_3} && {a_2}
	\arrow[tail reversed, from=1-1, to=1-3]
	\arrow[tail reversed, from=1-1, to=2-2]
	\arrow[tail reversed, from=1-3, to=2-2]
	\arrow[tail reversed, from=1-3, to=3-3]
	\arrow[tail reversed, from=3-1, to=1-1]
	\arrow[tail reversed, from=3-1, to=2-2]
	\arrow[tail reversed, from=3-3, to=2-2]
	\arrow[tail reversed, from=3-3, to=3-1]
\end{tikzcd}\]
\caption{Splitting the closed path $p$ into triangles.}
\label{fig:split-in-4}
\end{figure}
Now define $b = T_{a_1}(a_2)$. Using Proposition \ref{prop:dehn-int} it can be shown that $i(a_i,b) = 1$. Thus the path $p$ can be split into four paths of type 1, see Figure \ref{fig:split-in-4}. Hence $p$ is null-homotopic.
\end{proof}
\begin{lemma}\label{lem:radius-1}
  Every closed path $p=a_0\lr \cdots \lr a_{n-1} \lr a_0$ in $\Gamma_1(S)$ of radius $1$ with respect to $a_0$ is null-homotopic. 
\end{lemma}
\begin{proof}
  Since $p$ has radius $1$ with respect to $a_0$ it follows that $i(a_0, a_i) \leq 1$ for all $i$. In the case that $i(a_0, a_i) = 1$ for all $i>0$ it follows that the path $p$ can be split into $n-2$ paths of type 1, and hence conclude that $p$ is null-homotopic. Now consider the case when $i(a_0, a_j) = 0$ for exactly one $j\neq 0$ and $i(a_0,a_i) = 1$ for $i\neq 0,j$. In this case $p$ can be split into paths of type 1 and the closed path $q = a_0\lr a_{j-1}\lr a_j\lr a_{j+1} \lr a_0$. Since paths of type are null-homotopic and $q$ is null-homotopic by Proposition \ref{prop:four-cycles} it follows that $p$ is null homotopic.\\
  \begin{figure}[t]
    \[\begin{tikzcd}
	& {a_0} & {a_7} &&&& {a_0} & {a_7} \\
	{a_1} &&& {a_6} && {a_1} &&& {a_6} \\
	{a_2} &&& {a_5} && {a_2} &&& {a_5} \\
	& {a_3} & {a_4} &&&& {a_3} & {a_4} \\
	\\
	&&&& {a_0} & {a_7} \\
	&&& {a_1} &&& {a_6} \\
	&&& {a_2} &&& {a_5} \\
	&&&& {a_3} & {a_4}
	\arrow[tail reversed, from=1-2, to=2-1]
	\arrow[tail reversed, from=1-2, to=2-4]
	\arrow[tail reversed, from=1-2, to=3-1]
	\arrow[tail reversed, from=1-2, to=3-4]
	\arrow[tail reversed, from=1-2, to=4-2]
	\arrow[tail reversed, from=1-2, to=4-3]
	\arrow[tail reversed, from=1-3, to=1-2]
	\arrow[tail reversed, from=1-7, to=2-6]
	\arrow[tail reversed, from=1-7, to=2-9]
	\arrow[tail reversed, from=1-7, to=3-6]
	\arrow[tail reversed, from=1-7, to=3-9]
	\arrow[tail reversed, from=1-7, to=4-7]
	\arrow[tail reversed, from=1-8, to=1-7]
	\arrow[tail reversed, from=2-1, to=3-1]
	\arrow[tail reversed, from=2-4, to=1-3]
	\arrow[tail reversed, from=2-6, to=3-6]
	\arrow[tail reversed, from=2-9, to=1-8]
	\arrow[tail reversed, from=3-1, to=4-2]
	\arrow[tail reversed, from=3-4, to=2-4]
	\arrow[tail reversed, from=3-6, to=4-7]
	\arrow[tail reversed, from=3-9, to=2-9]
	\arrow[tail reversed, from=4-2, to=4-3]
	\arrow[tail reversed, from=4-3, to=3-4]
	\arrow[tail reversed, from=4-7, to=4-8]
	\arrow[tail reversed, from=4-8, to=3-9]
	\arrow[tail reversed, from=6-5, to=7-4]
	\arrow[tail reversed, from=6-5, to=7-7]
	\arrow[tail reversed, from=6-5, to=9-5]
	\arrow[tail reversed, from=6-6, to=6-5]
	\arrow[tail reversed, from=7-4, to=8-4]
	\arrow[tail reversed, from=7-7, to=6-6]
	\arrow[tail reversed, from=8-4, to=9-5]
	\arrow[tail reversed, from=8-7, to=7-7]
	\arrow[tail reversed, from=9-5, to=9-6]
	\arrow[tail reversed, from=9-6, to=8-7]
\end{tikzcd}\]
\caption{Examples of the three cases in proof of Lemma \ref{lem:radius-1}.}
\end{figure}
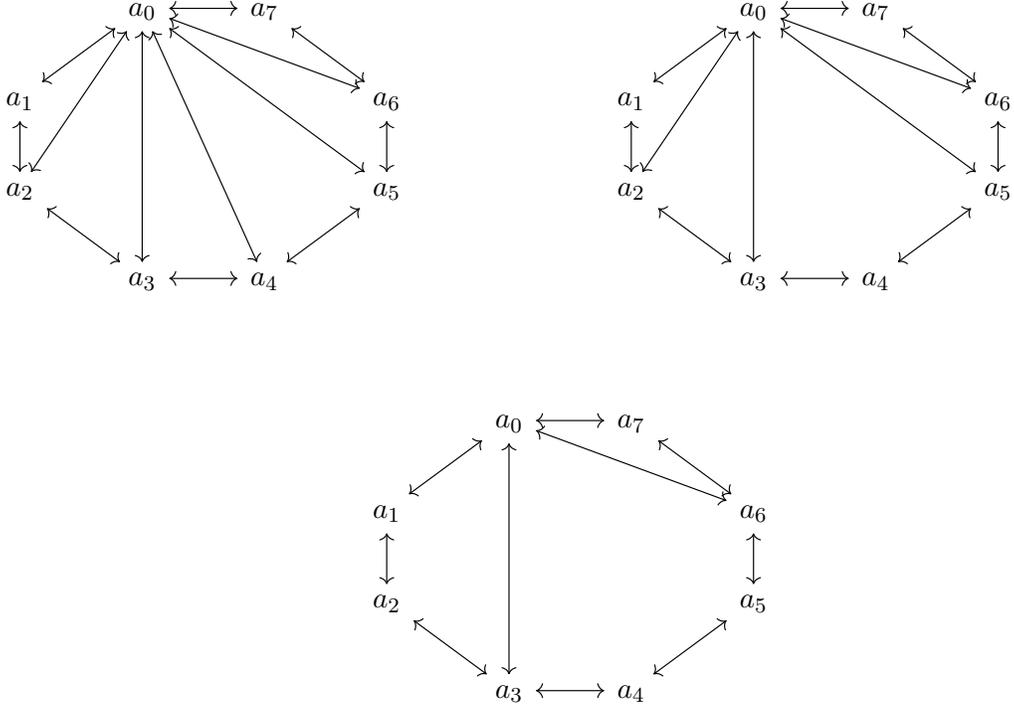

Finally, consider the case where the intersection number of $a_0$ with some $a_i$ are 0 and with the rest it is $1$. Then the path $p$ can be split into paths of type 1, paths of the type in Proposition \ref{prop:four-cycles}, and some larger paths. Paths of the first two types are null homotopic, by definition of $\Gamma_1(S)$ and Proposition \ref{prop:four-cycles} respectively, thus we only have to show that paths of the third type are null-homotopic. Let $q = a_0\lr a_j \lr a_{j+1}\lr \cdots \lr a_{j+m}\lr a_0$ be a path of the third type. Then we know that $i(a_0, a_{j+s})=0$ for all $0< s< m$. Using the same argument as in Proposition \ref{prop:four-cycles} we can assume without loss of generality that $i(a_j, a_{j+2})=0$. Let $b=T_{a_j}(a_{j+1})$, then using the intersection identity in \ref{prop:dehn-int} it follows that $i(b,a_0) = i(b,a_{j+2}) = i(b,a_{j+1}) = 1$. This means that $q$ is homotopic to the path $a_0\lr a_j \lr b \lr a_{j+2} \lr \cdots \lr a_{j+m}\lr a_0$ which, due to $a_0\lr a_j\lr b\lr a_0$ being a path of type 1, is homotopic to the path $q_1=a_0\lr b\lr a_{j+2}\lr \cdots \lr a_{j+m} \lr a_0$. Repeating this process we eventually arrive at a path $q_n$ which is of the type in Proposition \ref{prop:four-cycles}, which is null-homotopic. This completes the proof.
\end{proof}
\begin{proposition}\label{prop:base-case}
  Every closed path $p=a_0\lr\cdots\lr a_{n-1}\lr a_0$ in $\Gamma_1(S)$ is null-homotopic.
\end{proposition}
\begin{proof}
  The idea of the proof is to construct curves $\b_0,\ \b_1,\ \b_2$ in $S$ so that $i(b_2,a_j) = 0$ for all $j$ and $i(b_2,b_0) = i(b_2,b_1)= i(b_0,a_0) = i(b_1,a_1) = 1$. This means that the path $q= b_2 \lr b_1 \lr a_1 \lr \cdots \lr a_{n-1}\lr a_0 \lr b_0 \lr b_2$ has radius 1 with respect to $b_2$. Since the path $b_2\lr b_1\lr a_1 \lr a_0\lr b_0\lr b_2$ is also radius $1$ with respect to $b_2$ it follows from Lemma \ref{lem:radius-1} that $q$ is homotopic to $p$ and that $q$ is null homotopic. It follows that $p$ is null-homotopic.\\

  Now we describe the construction of the curves $\b_i$. Let $R$ be a subsurface of $S$ with compact boundary and at least one non-planar end so that all the curves $\a_0,\cdots, \a_{n-1}$ are contained in $R$. Let $\delta$ be a separating curve in $R$ disjoint from all $\a_i$ so that the two surfaces, $R_1$ and $R_2$, obtained by cutting $R$ along $\delta$ satisfy the following: $R_1$ is compact and contains all $\a_i$ and $R_2$ is of infinite-type with a single compact boundary component corresponding to the curve $\delta$ in $R$ and has non-planar ends. Let $\tau_0$ and $\tau_1$ be homotopically distinct arcs in $R_2$ with their end points lying on the boundary of $R_2$ so that $\tau_0\cap \tau_1 = \emptyset$. Since $\delta$ and $\a_0$ are disjoint and $i(a_0,a_1)=1$, it can be shown that there exist disjoint essential arcs, $\sigma_0,\ \sigma_1$, with end points on the boundary of $R_1$ corresponding to $\delta$ and $|\sigma_0\cap \a_0| = 1$ and $|\sigma_1\cap \a_1|=1$. Gluing $R_1$ and $R_2$ together along the boundary corresponding to $\delta$ in such a way that $\beta_0 = \sigma_0\cup \tau_0$ and $\beta_1 = \sigma_1\cup \tau_1$ are closed curves. Then by construction we have $i(b_0,a_0) = i(b_1,a_1) = 1$ and that $b_i$ are non-separating. Let $\b_2$ be a non-separating curve in $R_2$ which intersects $\tau_0$ and $\tau_1$ exactly once. Thus we have our desired curves $b_0,\ b_1,$ and $b_2$.
\end{proof}
The above proposition implies that $\Gamma_1(S)$ is simply-connected when $S$ has non-planar ends. We use this as our base case for induction on $k$, to show that $\Gamma_k(S)$ is also simply-connected. From here on we assume the induction hypothesis: for every surface $S$ the complex $\Gamma_j(S)$ is simply connected for all $j<k$. We prove the simple connectedness now by also inducting on radius of the path. We start by a special case of radius 0 curves.
\begin{proposition}\label{prop:sp-radius-0}
  Suppose that $p = v_0 \lr \cdots \lr v_{n-1}$ is a closed path in $\Gamma_k(S)$ so that all the vertices $v_i$ contain a common curve $c$. Then $p$ is null-homotopic.
\end{proposition}
\begin{proof}
  Let $p'$ be the path in $\Gamma_{k-1}(S)$ where the vertices are $v_i- \{c\}$. By the induction hypothesis, it follows that $p'$ is null-homotopic, meaning that it can be split into paths of type 1, 2, and 3. Since adding $c$ back to each vertex in $p'$ does not remove any edge connections therefore $p$ is null-homotopic.
\end{proof}
\begin{proposition}[Lemma 13, \cite{wajnryb}]\label{prop:hex-path}
  Let $a_0, a_1, a_2$ be disjoint non-separating curves so that:
  \begin{enumerate}
    \item all unions $\a_i\cup \a_j$ are non-separating in $S$.
    \item the union $\a_0\cup \a_1\cup \a_2$ separates $S$.
  \end{enumerate}
  Then there exists disjoint non-separating curves $b_0, b_1, b_2$ so that $i(a_i,b_i) = 1$, $i(a_{i+1}, b_{i}) = 1$ (indices modulo $3$), and $i(a_i, b_{i+1}) = 0$ for all other $i,j$. Moreover, if $(a_0,a_1,c_2,\cdots, c_{k-1})$ is a vertex in $\Gamma_k(S)$ then the closed path $p = v_0 \lr \cdots \lr v_5 \lr v_0$ defined as
  \begin{align*}
    (a_0, a_1, c_2,\cdots, c_{k-1})\lr &(a_0, b_1,c_2,\cdots, c_{k-1}) \lr (a_0, a_2,c_2,\cdots, c_{k-1}) \lr (b_0, a_2,c_2,\cdots, c_{k-1})\\ 
    &\ \lr (a_1,a_2,c_2,\cdots, c_{k-1})\lr (a_1, b_2,c_2,\cdots, c_{k-1})\lr (a_0, a_1,c_2,\cdots, c_{k-1})
  \end{align*}
  is null-homotopic in $\Gamma_k(S)$.
\end{proposition}
\begin{proof}
  Due to conditions on $a_0,\ a_1,\ a_2$ it follows that $S - \bigcup_{i}\a_i$ has exactly two connected components, $R_1$ and $R_2$, both with at least three boundary components corresponding to the curves $\a_0,\ \a_1,$ and $\a_2$. Let $\delta_i$ be disjoint arcs in $R_1$ with endpoints on the boundary components corresponding to $\a_i$ and $\a_{i+1}$ (indices taken modulo 3). Similarly define disjoint arcs $\delta'_i$ on $R_2$. Defining $\b_i = \delta_i\cup \delta'_i$ we get disjoint non-separating curves $b_0, b_1, b_2$ with the desired intersections with $a_i$. In order to prove the second part of the proposition, define $c = T_{a_1}(b_0)$. Then it follows from Proposition \ref{prop:dehn-int} that $i(c, b_0) = i(c, a_1) = i(c,b_1) = i(c,a_0) = 1$ and $i(c,b_2) = i(c, a_2) = 0$. If $(a_0, a_1, c_2,\cdots, c_{k-1})$ is a vertex in $\Gamma_k(S)$ then define the vertices $w_0 = (b_1,b_2, c_2,\cdots, c_{k-1}), w_1 = (c,b_2,c_2,\cdots, c_{k-1}), w_2 = (c,a_2,c_2,\cdots, c_{k-1})$. Then the path $p = v_0 \lr \cdots \lr v_5 \lr v_0$ as defined above can be split into paths of type 1, 2, and 3 as in Figure \ref{fig:hex-path}. 
  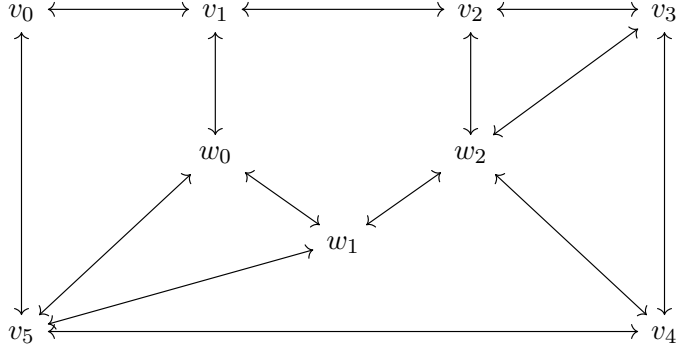
\begin{figure}
    \[\begin{tikzcd}
	{v_0} && {v_1} && {v_2} && {v_3} \\
	\\
	&& {w_0} && {w_2} \\
	&&& {w_1} \\
	{v_5} &&&&&& {v_4}
	\arrow[tail reversed, from=1-1, to=1-3]
	\arrow[tail reversed, from=1-3, to=1-5]
	\arrow[tail reversed, from=1-5, to=1-7]
	\arrow[tail reversed, from=1-7, to=5-7]
	\arrow[tail reversed, from=3-3, to=1-3]
	\arrow[tail reversed, from=3-3, to=4-4]
	\arrow[tail reversed, from=3-3, to=5-1]
	\arrow[tail reversed, from=3-5, to=1-5]
	\arrow[tail reversed, from=3-5, to=1-7]
	\arrow[tail reversed, from=3-5, to=5-7]
	\arrow[tail reversed, from=4-4, to=3-5]
	\arrow[tail reversed, from=4-4, to=5-1]
	\arrow[tail reversed, from=5-1, to=1-1]
	\arrow[tail reversed, from=5-7, to=5-1]
\end{tikzcd}\]
    \caption{Splitting the path $p$ into paths of type $1$, $2$, and $3$.}
    \label{fig:hex-path}
  \end{figure}
\end{proof}
Before beginning the proof of the next lemma, note that if $v$ and $w$ are vertices in $\Gamma_k(S)$ with common curves $(a_0, \cdots, a_m)$ then there is an $(a_0,\cdots, a_m)-$segment connecting $v$ and $w$. This follows directly from the proof of connectedness. This fact will be repeatedly used in the proof of the next lemma.
\begin{lemma}\label{lem:radius-zero}
  Let $p$ be a path in $\Gamma_k(S)$ of radius $0$ with respect to some curve $a_0$. Then $p$ is null-homotopic.
\end{lemma}
\begin{proof}
  We prove this by induction on the number of segments in the path $p$. Suppose that $a_0\in v_0$, where $v_0$ is a vertex in $p$. We call the maximal $a_0-$segment containing $v_0$ the first segment. From Proposition \ref{prop:sp-radius-0} it follows that if $p$ is a $a_0-$segment then $p$ is null homotopic. Suppose this is not the case and $v_1$ is the terminal vertex of the first segment. Let $a_1\in v_1$ be a curve distinct from $a_0$ which will serve as the common curve of the next segment, which we call the second segment. Let $v_2$ be the terminal element of second segment. If $p$ is just made up of these two segments then both $v_1$ and $v_2$ contain $a_0$ and $a_1$, thus there is a $(a_0, a_1)-$segment from $v_1$ to $v_2$, therefore implying that $p$ can be split into a $a_0-$segment and a $a_1-$segment, both of which are null-homotopic due to Proposition \ref{prop:sp-radius-0}. Suppose now that $p$ has more than $2$ segments. Then let $a_2$ be a curve in the vertex $v$ which follows $v_2$, so that $i(a_0, a_2)=0$. Let $v_3$ be the terminal vertex in the $a_2-$segment. Now we try to reduce the number of segments.\\

  \begin{figure}[h]
    \[\begin{tikzcd}
	\cdots & {v_0} && {v_1} && {v_2} & v && {v_3} & \cdots \\
	\\
	&&&&& {u_1} & {u_2}
	\arrow[color={rgb,255:red,214;green,92;blue,92}, tail reversed, from=1-1, to=1-2]
	\arrow["{a_0-\text{segment}}", color={rgb,255:red,214;green,92;blue,92}, squiggly, tail reversed, from=1-2, to=1-4]
	\arrow["{a_1-\text{segment}}", squiggly, tail reversed, from=1-4, to=1-6]
	\arrow["{q_0}"', color={rgb,255:red,214;green,92;blue,92}, squiggly, tail reversed, from=1-4, to=3-6]
	\arrow[tail reversed, from=1-6, to=1-7]
	\arrow["{q_1}"', squiggly, tail reversed, from=1-6, to=3-6]
	\arrow["{a_2-\text{segment}}", color={rgb,255:red,214;green,92;blue,92}, squiggly, tail reversed, from=1-7, to=1-9]
	\arrow["{q_2}", color={rgb,255:red,214;green,92;blue,92}, squiggly, tail reversed, from=1-7, to=3-7]
	\arrow[color={rgb,255:red,214;green,92;blue,92}, tail reversed, from=1-9, to=1-10]
	\arrow[color={rgb,255:red,214;green,92;blue,92}, tail reversed, from=3-6, to=3-7]
    \end{tikzcd}\]
    \caption{The construction in Case 1. Here $q_0$ is a $(a_0,a_1)-$segment, $q_1$ is a $a_1-$segment, and $q_2$ is a $a_2-$segment. The path in red has one less segment than $p$.}
    \label{fig:case1}
  \end{figure}
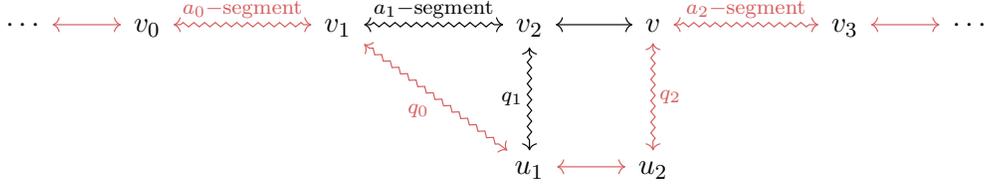
  \textbf{Case 1:} Suppose that $a_2\notin v_2$. Since $v_2\lr v$, it follows that $v-v_2 = \{a_2\}$ and $v_2-v = \{a_1\}$, implying that $i(a_1,a_2) = 1$. Let $a_1\cap a_2 = \{b_1,\cdots, b_{k_1}\}$. This means that we obtain a connected surface, $R$, by cutting $S$ along $a_1\cup a_2$. Moreover the image of the curve $a_0$ in $R$ is non-separating. Let $u = (b_1, \cdots, b_{k-1})$ and $u'$ be a vertex in $\Gamma_{k-1}(R)$ containing $a_0$. By connectedness, it follows that there is a path $q$ between $u$ and $u'$. Let $q_1, q_2$ be paths in $\Gamma_k(S)$ obtained by adding $a_1, a_2$ respectively to each vertex of $q$. The path $q_1$ connects $v_2$ to $u_1$ and $q_2$ connects $v$ to $u_2$, where $u_i$ contains both $a_i$ and $a_0$. Note that corresponding vertices $v', w'$ in paths $q_1, q_2$ respectively are connected by an edge as $i(a_1,a_2) = 1$. Thus, observe that if $v'\lr v''$ in $q_1$ and $w'\lr w''$ in $q_2$ then the path $v'\lr w' \lr w''\lr v'' \lr v'$ is a path of type 2. It follows from this that the closed path which starts at $v_2$ moves along $q_1$ till $u_1$, goes to $u_2$ via the connecting edge, $u_1\lr u_2$, then moves back up along $q_2$ to $v$ and returns back to $v_2$ via the connecting edge, is null-homotopic. Also since $v_1$ and $u_1$ contain curves $a_0, a_1$ there is a $(a_0, a_1)-$segment, $q_0$, connecting $v_1$ to $u_1$. The closed path which starts from $v_1$ and moves along $p$ till $v_2$, then moves along $q_1$ to $u_1$, and finally moves back up to $v_1$ along $q_0$ is also null-homotopic as it is a $a_1-$segment. Now, the red path, $p'$, in Figure \ref{fig:case1} is homotopic to $p$ and bypasses the $a_1-$segment in $p$, thus $p'$ has one less segment than $p$.\\

  \textbf{Case 2:} Suppose now that $a_2\in v_2$ and $\a_0 \cup \a_2$ is non-separating. In the case when $\a_0\cup \a_1\cup \a_2$ is non-separating then there is a vertex $v'\in \Gamma_k(S)$ so that $v'$ contains $a_0, a_1, a_2$. Then there is a $(a_0, a_1)-$segment from $v_1$ to $v'$ and a $(a_1,a_2)-$segment from $v_2$ and $v'$. These are shown in Figure \ref{fig:case2-1}.
  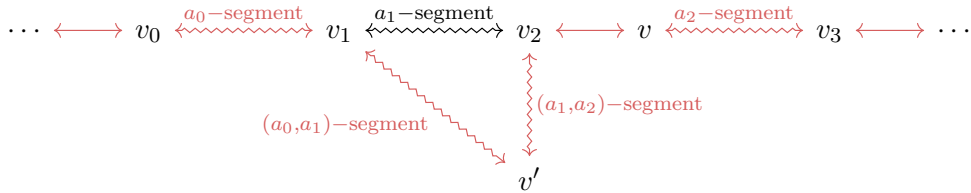
\begin{figure}[h]
   \[\begin{tikzcd}
	\cdots & {v_0} && {v_1} && {v_2} & v && {v_3} & \cdots \\
	\\
	&&&&& {v'}
	\arrow[color={rgb,255:red,214;green,92;blue,92}, tail reversed, from=1-1, to=1-2]
	\arrow["{a_0-\text{segment}}", color={rgb,255:red,214;green,92;blue,92}, squiggly, tail reversed, from=1-2, to=1-4]
	\arrow["{a_1-\text{segment}}", squiggly, tail reversed, from=1-4, to=1-6]
	\arrow["{(a_0,a_1)-\text{segment}}"', color={rgb,255:red,214;green,92;blue,92}, squiggly, tail reversed, from=1-4, to=3-6]
	\arrow[color={rgb,255:red,214;green,92;blue,92}, tail reversed, from=1-6, to=1-7]
	\arrow["{a_2-\text{segment}}", color={rgb,255:red,214;green,92;blue,92}, squiggly, tail reversed, from=1-7, to=1-9]
	\arrow[color={rgb,255:red,214;green,92;blue,92}, tail reversed, from=1-9, to=1-10]
	\arrow["{(a_1,a_2)-\text{segment}}"', color={rgb,255:red,214;green,92;blue,92}, squiggly, tail reversed, from=3-6, to=1-6]
   \end{tikzcd}\]
   \caption{The red path is homotopic to $p$.}
   \label{fig:case2-1}
  \end{figure}
  The red path in Figure \ref{fig:case2-1} is homotopic to $p$ as the closed path $v_1\leftrightsquigarrow v'\leftrightsquigarrow v_2 \leftrightsquigarrow v_1$ is null-homotopic by Proposition \ref{prop:sp-radius-0}. The red path has fewer segments than $p$ as we have bypassed the second segment.\\

  Now consider the case when the union $\a_0\cup \a_1 \cup \a_2$ is separating. Let $(b_2,\cdots, b_{k-1})$ be a cut system in the component of $S-(\a_0\cup\a_1\cup\a_2)$ which has a non-planar end. Then let $w_0 = (a_0, a_1, b_2,\cdots, b_{k-1}),\ w_1 = (a_1, a_2, b_2,\cdots, b_{k-1}),$ and $w_2 = (a_2, a_0, b_1,\cdots, b_{k-1})$ be vertices in $\Gamma_k(S)$. By Proposition \ref{prop:hex-path}, there is a null-homotopic closed path $q=w_0\squig w_1 \squig w_2 \squig w_0$ where the path $w_i \squig w_{i+1}$ is a $a_{i+1}-$segment (indices modulo 3) for $i=0,1,2$. Moreover, there is a $(a_0,a_1)-$segment, $q_1$, from $v_1$ to $w_0$ and a $(a_1, a_2)-$segment, $q_2$, from $v_2$ to $w_1$. The closed path $v_1\squig w_0 \squig w_1 \squig v_2 \squig v_1$ as shown in Figure \ref{fig:case2-2} is also null-homotopic as it is a $a_1-$segment, by Proposition \ref{prop:sp-radius-0}. It follows that the red is homotopic to $p$ and has fewer segments than $p$.\\

  \begin{figure}[h]
    \[\begin{tikzcd}
	\cdots & {v_0} && {v_1} && {v_2} & v && {v_3} & \cdots \\
	\\
	&&& {w_0} && {w_1} \\
	&&&& {w_2}
	\arrow[color={rgb,255:red,214;green,92;blue,92}, tail reversed, from=1-1, to=1-2]
	\arrow["{a_0-\text{segment}}"', color={rgb,255:red,214;green,92;blue,92}, squiggly, tail reversed, from=1-4, to=1-2]
	\arrow["{a_1-\text{segment}}", squiggly, tail reversed, from=1-4, to=1-6]
	\arrow["{(a_0, a_1)-\text{segment}}"', color={rgb,255:red,214;green,92;blue,92}, squiggly, tail reversed, from=1-4, to=3-4]
	\arrow[color={rgb,255:red,214;green,92;blue,92}, tail reversed, from=1-6, to=1-7]
	\arrow["{(a_1, a_2)-\text{segment}}", color={rgb,255:red,214;green,92;blue,92}, squiggly, tail reversed, from=1-6, to=3-6]
	\arrow["{a_2-\text{segment}}", color={rgb,255:red,214;green,92;blue,92}, squiggly, tail reversed, from=1-7, to=1-9]
	\arrow[color={rgb,255:red,214;green,92;blue,92}, tail reversed, from=1-9, to=1-10]
	\arrow["{a_1-\text{segment}}", squiggly, tail reversed, from=3-4, to=3-6]
	\arrow["{a_0-\text{segment}}"', color={rgb,255:red,214;green,92;blue,92}, squiggly, tail reversed, from=3-4, to=4-5]
	\arrow["{a_2-\text{segment}}"', color={rgb,255:red,214;green,92;blue,92}, squiggly, tail reversed, from=4-5, to=3-6]
    \end{tikzcd}\]
    \caption{The red path is homotopic to $p$.}
    \label{fig:case2-2}
  \end{figure}
  
  \textbf{Case 3:} Now suppose that $a_2\in v_2$ and $\a_0 \cup \a_2$ separates $S$. Suppose that $p$ has only three segments. The first vertex of the of first segment then contains both $a_0$ and $a_2$ as $i(a_0, a_2) = 0$. But this implies that $\a_0\cup \a_2$ is non-separating, which is a contradiction. Thus $p$ must contain at least four segments in this case. We assume from here on that closed paths with radius 0 and up to three segments are null homotopic. Let $a_3$ be a non-separating curve in the vertex following $v_3$ so that $i(a_0, a_3) = 0$. Note that $i(a_2, a_3) \neq 1$ as $\a_2$ is a separating curve in the surface obtained by cutting $S$ along $\a_0$. Since $v_3$ is the last vertex of the third segment and $a_3$ cannot be involved in the move from $v_3$ to the next vertex, it follows that $a_3\in v_3$. This also implies that $\a_2\cup \a_3$ is non-separating. Now, let $S_0, S_1$ be the two surfaces obtained by cutting $S$ along $\a_0 \cup \a_2$ and without loss of generality assume that $\a_4$ is a curve in $S_1$. Note that $\a_0\cup \a_3$ is non-separating in $S$ since otherwise $\a_2\cup \a_3$ is separating in $S$, a contradiction. Since $i(a_1,a_0) = i(a_1, a_2) = 0$ we can break this further down into two cases: either $\a_1$ is contained in $S_0$ or $\a_1$ is contained in $S_1$.\\

  Suppose that $\a_1$ is in $S_0$. Then $\a_0\cup \a_1 \cup \a_3$ is non-separating in $S$ as $\a_1$ is non-separating in $S_0$ and $\a_3$ is non-separating in $S_1$. This means that there is a vertex $w = (a_0, a_1, a_3, b_3 ,\cdots, b_{k-1})$ in $\Gamma_k(S)$. These segments from $v_1,\ v_2,\ v_3$ to $w$ as shown in Figure \ref{fig:case3-1}. The red path depicted in Figure \ref{fig:case3-1} is homotopic to $p$ by Proposition \ref{prop:sp-radius-0} and has fewer segments than $p$.\\
  \begin{figure}[h]
    \[\begin{tikzcd}
	\cdots & {v_0} && {v_1} && {v_2} & v && {v_3} & \cdots \\
	\\
	&&&& w
	\arrow[color={rgb,255:red,214;green,92;blue,92}, tail reversed, from=1-1, to=1-2]
	\arrow["{a_0-\text{segment}}"', color={rgb,255:red,214;green,92;blue,92}, squiggly, tail reversed, from=1-4, to=1-2]
	\arrow["{a_1-\text{segment}}", squiggly, tail reversed, from=1-4, to=1-6]
	\arrow["{(a_0, a_1)-\text{segment}}"', color={rgb,255:red,214;green,92;blue,92}, squiggly, tail reversed, from=1-4, to=3-5]
	\arrow[tail reversed, from=1-6, to=1-7]
	\arrow["{a_2-\text{segment}}", squiggly, tail reversed, from=1-7, to=1-9]
	\arrow[color={rgb,255:red,214;green,92;blue,92}, tail reversed, from=1-9, to=1-10]
	\arrow["{a_1-\text{segment}}"{description, pos=0.5}, squiggly, tail reversed, from=3-5, to=1-6]
	\arrow["{a_3-\text{segment}}"', color={rgb,255:red,214;green,92;blue,92}, squiggly, tail reversed, from=3-5, to=1-9]
\end{tikzcd}\]
    \caption{The red path is homotopic to $p$.}
    \label{fig:case3-1}
  \end{figure}

  Now suppose that $\a_1$ is contained in $S_1$. If $S_1$ is a surface of finite-type with genus $g$ then let $w,w'$ be maximal cut systems in $ \mathcal{HT}(S_1)$ containing $a_1, a_3$ respectively and let $t\in \Gamma_{k-g-1}(S_0)$ be some fixed vertex. If $S_1$ is infinite-type then let $w,w'\in \Gamma_{k-1}(S_1)$ be vertices containing $a_1, a_3$ respectively and let $t$ be empty. In either case, from connectivity, it follows that there is a path $q$ in $\Gamma_{k-1}(S)$ from $w$ to $w'$. Let $q_1$ be the path in $\Gamma_k(S)$ obtained by adding $a_0$ and the curves in $t$ to each vertex in $q$, and let $q_2$ be the curve obtained by adding $a_3$ and the curves in $t$ to each vertex in $q$. The end points, $w_2$ and $w'_2$, of $q_1$ contain the curves $a_0, a_1$ and $a_0, a_3$ respectively. Similarly there are end points, $w_3$ and $w'_3$, of $q_2$ that contains the curves $a_2, a_1$ and $a_2, a_3$ respectively. Figure \ref{fig:case3-2} shows segments connecting $v_i$ to $w_i$ for $i=2,3$, $v_1$ to $w'_2$, and $w'_3$ to $v_4$. The closed paths $v_1 \squig v_2 \squig w_2 \squig w'_2 \squig v_1$, $v_2\squig w_2 \squig w_3 \squig v_3 \squig v_2$, and $v_4 \squig v_3 \squig w_3 \squig w'_3 \squig v_4$ are null homotopic as they have zero radius and have at most three segments. Connecting corresponding edges in $q_1$ and $q_2$ by a path, we can split the closed path $w_2 \squig w_3 \squig w'_3 \squig w'_2 \squig w_2$ into closed paths with $a_0-$radius 0 and consisting only 3 segments, therefore concluding that this closed path is also null-homotopic. Thus the red path depicted in Figure \ref{fig:case3-2} is homotopic to $p$ and has fewer segments.
  \begin{figure}[h]
    \[\begin{tikzcd}
	\cdots & {v_0} && {v_1} && {v_2} && {v_3} && {v_4} & \cdots \\
	\\
	&&&&& {w_2} && {w_3} \\
	\\
	&&&&& {w_2'} && {w_3'}
	\arrow[color={rgb,255:red,214;green,92;blue,92}, tail reversed, from=1-1, to=1-2]
	\arrow["{a_0-\text{segment}}"', color={rgb,255:red,214;green,92;blue,92}, squiggly, tail reversed, from=1-4, to=1-2]
	\arrow["{a_1-\text{segment}}", squiggly, tail reversed, from=1-4, to=1-6]
	\arrow["{a_0-\text{segment}}"', color={rgb,255:red,214;green,92;blue,92}, curve={height=30pt}, squiggly, tail reversed, from=1-4, to=5-6]
	\arrow["{a_2-\text{segment}}", squiggly, tail reversed, from=1-6, to=1-8]
	\arrow["{a_1-\text{segment}}"', squiggly, tail reversed, from=1-6, to=3-6]
	\arrow["{a_3-\text{segment}}", squiggly, tail reversed, from=1-8, to=1-10]
	\arrow["{a_2-\text{segment}}", squiggly, tail reversed, from=1-8, to=3-8]
	\arrow[color={rgb,255:red,214;green,92;blue,92}, tail reversed, from=1-10, to=1-11]
	\arrow["{a_3-\text{segment}}", color={rgb,255:red,214;green,92;blue,92}, curve={height=-30pt}, squiggly, tail reversed, from=1-10, to=5-8]
	\arrow["{a_1-\text{segment}}", squiggly, tail reversed, from=3-6, to=3-8]
	\arrow[""{name=0, anchor=center, inner sep=0}, "{a_0-\text{segment}}"'{pos=0.4}, squiggly, tail reversed, from=3-6, to=5-6]
	\arrow[""{name=1, anchor=center, inner sep=0}, "{a_2-\text{segment}}", squiggly, tail reversed, from=3-8, to=5-8]
	\arrow["{a_3-\text{segment}}"', color={rgb,255:red,214;green,92;blue,92}, squiggly, tail reversed, from=5-6, to=5-8]
\end{tikzcd}\]
  \caption{The red path is homotopic to $p$.}
  \label{fig:case3-2}
  \end{figure}
\end{proof}
Now we are finally ready to prove Theorem \ref{thm:simply}.
\begin{proof}[Proof of Theorem \ref{thm:simply}]
  Suppose that $p = v_0 \lr \cdots \lr v_{n-1}\lr v_0$ is a path in $\Gamma_k(S)$. Let $K$ be a compact set containing all the curves in the vertices of $p$ and let $b$ be a non-separating curve in $S-K$. Let $a_0\in v_0$ and $\ a_1\in v_1$ be disjoint curves and let $w_0$ and $w_1$ be vertices in $\Gamma_k(S)$ which contain $b, a_0$ and $b,a_1$ respectively. The vertices $w_0$ and $v_0$ can be connected by a $a_0-$segment, $w_1$ and $v_1$ can be connected by a $a_1-$segment, and $w_0$ and $w_1$ can be connected by a $b-$segment. The closed path $v_0 \lr w_0 \lr w_1 \lr v_1 \lr v_0$ has $a_0-$radius 0 implying that the red path in Figure \ref{fig:simply} is homotopic to $p$. Since the red path has $b-$radius 0 it follows that it is null-homotopic and thus $p$ is null-homotopic.
  \begin{figure}[h]
    \[\begin{tikzcd}
	\cdots & {v_0} & {v_1} & \cdots \\
	\\
	{w_0} &&& {w_1}
	\arrow[color={rgb,255:red,214;green,92;blue,92}, tail reversed, from=1-1, to=1-2]
	\arrow[tail reversed, from=1-2, to=1-3]
	\arrow["{a_0-\text{segment}}"'{pos=0.6}, color={rgb,255:red,214;green,92;blue,92}, squiggly, tail reversed, from=1-2, to=3-1]
	\arrow[draw={rgb,255:red,214;green,92;blue,92}, tail reversed, from=1-3, to=1-4]
	\arrow["{a_1-\text{segment}}"{pos=0.6}, color={rgb,255:red,214;green,92;blue,92}, squiggly, tail reversed, from=1-3, to=3-4]
	\arrow["{b-\text{segment}}"{description}, color={rgb,255:red,214;green,92;blue,92}, squiggly, tail reversed, from=3-1, to=3-4]
\end{tikzcd}\]
    \caption{The red path is homotopic to $p$.}
    \label{fig:simply}
  \end{figure}
\end{proof}
\section{Automorphisms of $\Gamma_k(S)$ and Rigidity}\label{aut}
In this section we prove Theorem \ref{thm:geometric} and henceforth assume that $S$ has no planar ends and boundary. We state the results proved in \cite{irmak2004} without proof whenever the proof can be directly translated to the infinite-type case without changes. Using Theorem \ref{thm:her} we already have that $\Aut(\Gamma_1(S)) \cong \map^*(S)$, as $\Gamma^1_1(S)= \mathcal{G}(S)$. First we describe an action of $\Aut(\Gamma_k(S))$ for $k\geq 2$ on the Schmutz graph, $\mathcal{G}(S)$.
\vspace{0.5em}
\newline
\textbf{Action of $\text{Aut}(\Gamma_k(S))$ on $\mathcal{G}(S)$.} Let $a$ be a non-separating curve and $f\in \Aut(\Gamma_k(S))$. Let $a_1,\cdots, a_{k-1}$ be curves in $S$ so that $v = (a,a_1,\cdots, a_{k-1})$ is a cut-system. Let $b$ be another non-separating curve with $i(a,b) = 1$ so that $w = (b,a_1, \cdots, a_{k-1})$ is also a cut-system. By construction, $v \lr w$ and since $f$ is simplicial it follows that $f(v) \lr f(w)$. Define $\tilde{f}(a)$ as the unique curve in $f(v) - f(w)$. The following lemma states that $\tilde{f}$ is well defined.
\begin{lemma}[Lemma 5, \cite{irmak2004}]
  The curve $\tilde{f}(a)$ is independent of all choices. $\square$
\end{lemma}
The proof of the above lemma involves several steps. First, for a given non-separating curve $a$ the subcomplex, $X_a\subset \mathcal{G}(S)$, whose vertices are curves which intersect $a$ exactly once is shown to be connected following the proof of Lemma 3 in \cite{irmak2004}. This fact along with the first half of Theorem \ref{thm:simply} is used to show that $\tilde{f}(a)$ is independent of the choices of $b$ and $a_1,\cdots, a_{k-1}$. Further, it is shown that $\tilde{f}$ is an automorphism of $ \mathcal{G}(S)$.
\begin{proposition}[Proposition 8, \cite{irmak2004}] \label{prop:simp}
  The map $\tilde{f}: \mathcal{G}(S) \to \mathcal{G}(S)$ is an automorphism. 
\end{proposition}
\begin{proof}
  Let $a,b$ be non-separating curves so that $i(a,b)=1$. There exists a collection of curves $a_1, \cdots, a_{k-1}$ which are disjoint from both $a$ and $b$ such that $v =(a, a_1, \cdots, a_{k-1})$ and $w = (b, a_1, \cdots, a_{k-1})$ are cut systems. Thus $f(v) \leftrightarrow f(w)$ as $f\in \Aut(\Gamma_k(S))$ and $v\leftrightarrow w$. Since by definition $\tilde{f}(a) = f(v)-f(w)$ and $\tilde{f}(b) = f(w) - f(v)$ it follows that $i(\tilde{f}(a), \tilde{f}(b)) = 1$. This shows that $\tilde{f}$ is a simplicial map. The inverse map to $\tilde{f}$ is given by $\widetilde{f^{-1}}$, see Lemma 7 in \cite{irmak2004}.
\end{proof}
Next, we generalize Proposition 9 in \cite{irmak2004}. Since $\tilde{f}$ can also be viewed as a map on the non-separating complex $ \mathcal{N}(S)$, it is natural to ask whether it is an automorphism of $ \mathcal{N}(S)$? The next proposition provides a positive answer.
\begin{proposition}
  The map $\tilde{f}: \mathcal{N}(S)\to \mathcal{N}(S)$ is also an automorphism.
\end{proposition}
\begin{proof}
  Using the same argument as before $\tilde{f}$ is a bijection. It only remains to show that it is also simplicial. Let $a,b\in \mathcal{N}(S)$ be two vertices so that $a\lr b$, or in other words $i(a,b)=0$. Consider the case when $\a\cup \b$ is non-separating in $S$. Let $a_2, \cdots a_{k-1}$ be non-separating curves so that $v = (a,b,a_2,\cdots, a_{k-1})$ is a vertex in $\Gamma_k(S)$. Since $\tilde{f}(a)$ and $\tilde{f}(b)$ are in $f(v)$ by definition it follows that $i(\tilde{f}(a), \tilde{f}(b)) = 0$.\\

  Now suppose that $\a \cup \b$ is separating in $S$. Then without loss of generality we can assume that they are as shown in the Figure \ref{fig:big_surface}. We construct a system of curves on $S$ in the following way:
  \begin{enumerate}
    \item Let $\{K_n\}_{n\in \N}$ be an exhaustion of $S$ by compact subsurfaces with $K_1$ being a genus $1$ subsurface with two boundary components containing the curves $\a$ and $\b$ and each component of $K_n - K_{n-1}$ being a genus $1$ surface with $b\geq 2$ boundary components. It is always possible to construct such an exhaustion by choosing an appropriate family of separating curves, see Figure \ref{fig:big_surface}. Let $C_1$ be the system of curves and arcs in $K_1$ as shown in Figure \ref{fig:small_surface}. Let $\g_0$ be the curve which intersects $\a$ and $\b$ exactly once.
    \item Let $\Sigma_{1,b}$ be the topological type of a component, $X$, of $K_{n}- K_{n-1}$. Let $C_X$ be a chain of filling non-separating curves and essential arcs in $X$, as shown in the Figure \ref{fig:small_surface}. Let $C_n = \bigcup_X C_X$ be a system of curves on $K_n-K_{n-1}$ for $n\geq 2$.
    \item Let $C$ be a chain of curves in $S$ consisting of all the curves in $C_n$ for every $n$ and the curves obtained by joining the two arcs which meet at every component of $\bigcup_n\partial K_n$. See Figure \ref{fig:big_surface} for an example.
  \end{enumerate}
  For any two curves $\g, \g'\in C$ either $i(c, c') = 0$ or $i(c,c') = 1$. Since $\g\cup \g'$ is non-separating in $S$, the non-separating case implies that $i(\tilde{f}(c), \tilde{f}(c'))=0$ if $i(c,c')=0$ and it follows from Proposition \ref{prop:simp} that $i(\tilde{f}(c), \tilde{f}(c')) = 1$ if $i(c,c')=1$. Similarly, when $\g\in C$ is such that $i(a,c) = i(b,c) = 0$ it follows that $i(\tilde{f}(a), \tilde{f}(c)) = i(\tilde{f}(b), \tilde{f}(c)) = 0$ and $i(\tilde{f}(a), \tilde{f}(c_0)) = i(\tilde{f}(b), \tilde{f}(c_0)) = 1$. Let $d = \tilde{f}(c)$ for each $\g \in C$ and let $d_0 = \tilde{f}(c_0)$. Then, as discussed above, the collection $\tilde{f}(C)$ forms a chain of non-separating curves in $S$. Note that $\tilde{f}(C)$ is also a chain of non-separating curves in $S$ due to $\tilde{f}$ being an automorphism of $ \mathcal{G}(S)$. Cutting $S$ along $\tilde{f}(C)-d_0$ we get disjoint copies of an annuli $A$ with a puncture in each boundary. Since $\tilde{f}(a),\ \tilde{f}(b)$ are disjoint from every $d\neq d_0$ it follows that they correspond to curves in some copy of the annuli $A$. Since $\tilde{f}(a)\neq \tilde{f}(b)$ it follows that they lie in distinct copies of $A$ and thus are disjoint in $S$.
\end{proof}
\begin{figure}[t]
\centering
\def\svgwidth{\textwidth}
\begin{subfigure}{0.5\textwidth}
  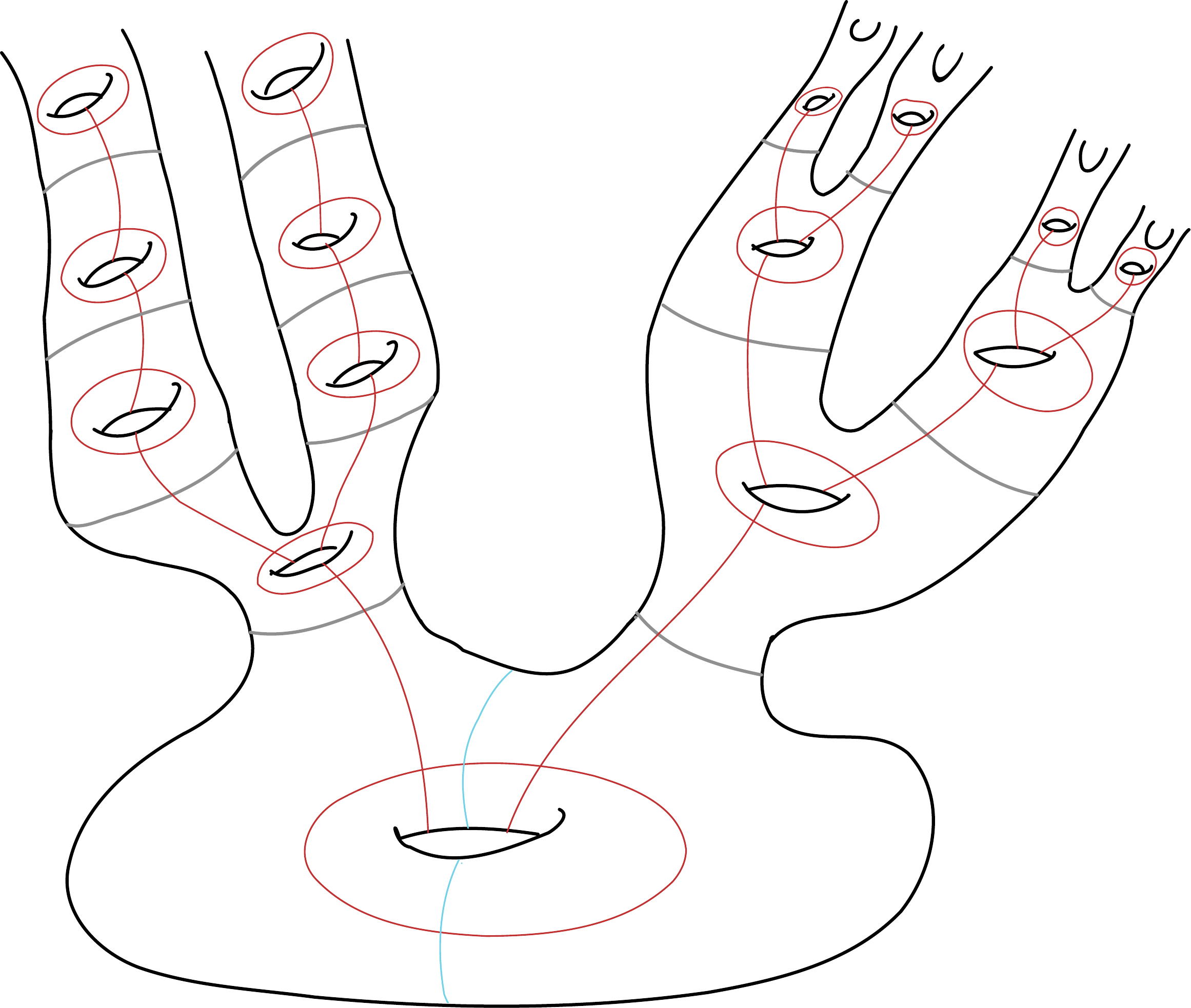
  \caption{This figure shows an example of the construction of $C$ on a surface with non-planar end space being the Cantor set union two isolated points. The blue curves are $\a$ and $\b$, the gray curves are boundaries of $\partial K_n$, and the red curves are elements of $C$.}
  \label{fig:big_surface}
\end{subfigure}
\hfill
\def\svgwidth{0.7\textwidth}
\begin{subfigure}{0.3\textwidth}
  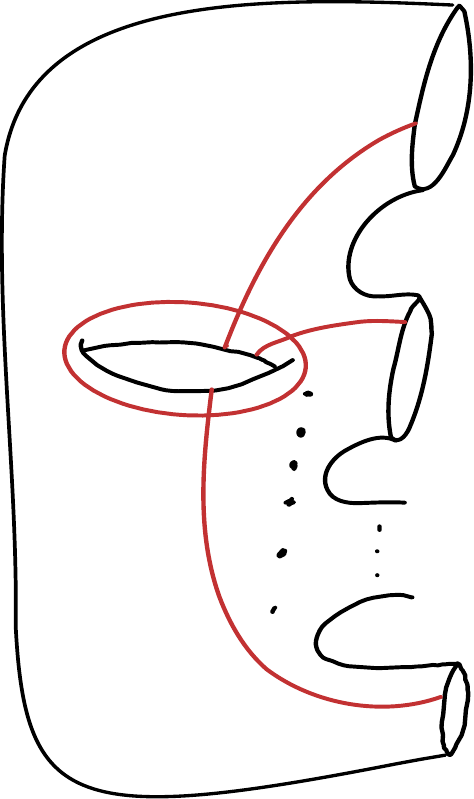
  \caption{The choice of chain on a surface of genus $1$ and $b$ boundary components.}
  \label{fig:small_surface}
\end{subfigure}
\caption{Example of constructions of $C_X$ and $C$.}
\end{figure}
\begin{proof}[Proof of Theorem \ref{thm:geometric}]
 Let $\Phi: \Aut(\Gamma_k(S)) \to \Aut(\mathcal{G}(S))$ be the map $f \mapsto \tilde{f}$. Suppose that $f,g\in \Aut(\Gamma_k(S))$ then by Lemma 7 in \cite{irmak2004} it follows that $\Phi(f\circ g) = \Phi(f)\circ \Phi(g)$. Suppose that $f\in \text{Ker}(\Phi)$ and let $v = (a_0,\cdots, a_{k-1})$ be a vertex in $\Gamma_k(S)$ then by construction $\tilde{f}(a_i) \in f(v)$ for all $i$, and thus $f(v) = (\tilde{f}(a_0), \cdots, \tilde{f}(a_{k-1})) = (a_0, \cdots, a_{k-1}) = v$. This proves that $\Phi$ is a monomorphism. Moreover, $\Phi$ is surjective since if $h\in \Aut(\mathcal{G}(S))$, then by Theorem \ref{thm:her} it follows that $h$ is induced from a homeomorphism $F:S\to S$. Let $f = F_*$ be the induced automorphism on $\Gamma_k(S)$. If $a$ is a non-separating curve, $v = (a, a_1, \cdots, a_{k-1})$, and $w = (b, a_1, \cdots, a_{k-1})$ are vertices in $\Gamma_k(S)$ as defined earlier then the unique curve in $f(v) - f(w)$ is just given by $F(a) = h(a)$ as $f,h$ are induced by $F$. Thus $\Phi$ is an epimorphism. Hence we have $\Aut(\Gamma_k(S)) \cong \Aut( \mathcal{G}(S)) \cong \map^*(S)$. This completes the proof of Theorem \ref{thm:geometric}.  
\end{proof}

Note that the map $\map^*(S) \to \Aut(\Gamma_k(S))$ induced by the natural action of $\map^*(S)$ on $\Gamma_k(S)$ factors through the natural map $\map^*(S) \to \Aut( \mathcal{G}(S))$ via $\Phi^{-1}$. In particular, the map $\map^*(S)\to \Aut(\Gamma_k(S))$ induced by the natural action is an isomorphism. Call this isomorphism $\Psi: \map^*(S)\to \Aut(\Gamma_k(S))$.

\begin{proof}[Proof of Corollary \ref{cor:iso}]
If $S$ and $S'$ are infinite-type surfaces and let $h: \Gamma_k(S)\to \Gamma_k(S')$ be an isomorphism. Then the map, $h_*: \Aut(\Gamma_k(S))\to \Aut(\Gamma_k(S'))$ induced by conjugation, i.e. $f\mapsto h\circ f\circ h^{-1}$, is an isomorphism. The composition $\Psi^{-1}\circ h_*\circ \Psi: \map^*(S) \to \map^*(S')$ is an isomorphism which, by Theorem 1.1 of \cite{bavard}, is induced by composition by a homeomorphism $G:S\to S'$, i.e. $\Psi^{-1}\circ h_*\circ \Psi: [F] \mapsto [G\circ F\circ G^{-1}]$. Thus we have $h_*(F_*) = G_*\circ F_*\circ G^{-1}_*$, for every mapping class $[F]$. Since by Theorem \ref{thm:geometric} every automorphism of $\Gamma_k(S)$ is induced by a mapping class we conclude that $h_*(f) = G_* \circ F_* \circ G^{-1}_*$ where $f = F_*$. This completes the proof of Corollary \ref{cor:iso}.
\end{proof}
\bibliography{refs.bib}
\vspace{5em}
\textsc{Correspondence Email:} \texttt{manvendra.somvanshi@pm.me}
\end{document}

%% file: big_surface.pdf_tex
\begingroup%
  \makeatletter%
  \providecommand\color[2][]{%
    \errmessage{(Inkscape) Color is used for the text in Inkscape, but the package 'color.sty' is not loaded}%
    \renewcommand\color[2][]{}%
  }%
  \providecommand\transparent[1]{%
    \errmessage{(Inkscape) Transparency is used (non-zero) for the text in Inkscape, but the package 'transparent.sty' is not loaded}%
    \renewcommand\transparent[1]{}%
  }%
  \providecommand\rotatebox[2]{#2}%
  \newcommand*\fsize{\dimexpr\f@size pt\relax}%
  \newcommand*\lineheight[1]{\fontsize{\fsize}{#1\fsize}\selectfont}%
  \ifx\svgwidth\undefined%
    \setlength{\unitlength}{1114.10054273bp}%
    \ifx\svgscale\undefined%
      \relax%
    \else%
      \setlength{\unitlength}{\unitlength * \real{\svgscale}}%
    \fi%
  \else%
    \setlength{\unitlength}{\svgwidth}%
  \fi%
  \global\let\svgwidth\undefined%
  \global\let\svgscale\undefined%
  \makeatother%
  \begin{picture}(1,0.8468954)%
    \lineheight{1}%
    \setlength\tabcolsep{0pt}%
    \put(0,0){\includegraphics[width=\unitlength,page=1]{big_surface.pdf}}%
  \end{picture}%
\endgroup%

%% file: finite_surface.pdf_tex
\begingroup%
  \makeatletter%
  \providecommand\color[2][]{%
    \errmessage{(Inkscape) Color is used for the text in Inkscape, but the package 'color.sty' is not loaded}%
    \renewcommand\color[2][]{}%
  }%
  \providecommand\transparent[1]{%
    \errmessage{(Inkscape) Transparency is used (non-zero) for the text in Inkscape, but the package 'transparent.sty' is not loaded}%
    \renewcommand\transparent[1]{}%
  }%
  \providecommand\rotatebox[2]{#2}%
  \newcommand*\fsize{\dimexpr\f@size pt\relax}%
  \newcommand*\lineheight[1]{\fontsize{\fsize}{#1\fsize}\selectfont}%
  \ifx\svgwidth\undefined%
    \setlength{\unitlength}{227.0794298bp}%
    \ifx\svgscale\undefined%
      \relax%
    \else%
      \setlength{\unitlength}{\unitlength * \real{\svgscale}}%
    \fi%
  \else%
    \setlength{\unitlength}{\svgwidth}%
  \fi%
  \global\let\svgwidth\undefined%
  \global\let\svgscale\undefined%
  \makeatother%
  \begin{picture}(1,1.68780156)%
    \lineheight{1}%
    \setlength\tabcolsep{0pt}%
    \put(0,0){\includegraphics[width=\unitlength,page=1]{finite_surface.pdf}}%
  \end{picture}%
\endgroup%